\documentclass[12pt,a4paper,oneside,reqno]{article}

\usepackage[utf8]{inputenc}
\usepackage[english]{babel}
\usepackage[left=1in,right=1in,top=1in,bottom=1in]{geometry}
\usepackage{amsfonts,amssymb,amsmath,amsthm}
\usepackage[symbol]{footmisc}
\usepackage{longtable,array}
\newcolumntype{L}[1]{>{\raggedright\arraybackslash$}p{#1}<{$}}
\usepackage{enumerate}
\usepackage{cite}
\usepackage{xcolor}

\usepackage{hyperref}
\hypersetup{hidelinks}
\usepackage{fancyhdr}
\theoremstyle{plain}
\newtheorem{theorem}{Theorem}

\newtheorem{proposition}[theorem]{Proposition}
\newtheorem{corollary}[theorem]{Corollary}
\theoremstyle{definition}

\newtheorem{definition}[theorem]{Definition}

\makeatletter
\@removefromreset{theorem}{section}\@removefromreset{equation}{section}

\makeatother
\usepackage{adjustbox}
\usepackage{tikz}\usetikzlibrary{arrows}

\IfFileExists{fouriernc.sty}{\usepackage{fouriernc}}{}

\newcommand{\doi}[1]{\href{https://doi.org/#1}{doi:\detokenize{#1}}}
\newcommand{\arxiv}[1]{\href{https://arxiv.org/abs/#1}{arXiv:#1}}

\usepackage{multicol,pifont,float,mathrsfs,comment}
\usepackage{blkarray, bigstrut}
\theoremstyle{plain}

\newtheorem{Lem}[theorem]{Lemma}
\newtheorem{prop}[theorem]{Proposition}
\theoremstyle{definition}

\newtheorem*{question}{Open question}

\newcommand{\F}{{\mathbb{F}}}

\newcommand{\cmark}{\text{\ding{51}}}
\newcommand{\xmark}{\text{\ding{55}}}

\def\remove#1{}

\begin{document}

\noindent\fbox{\parbox{\dimexpr\textwidth-2\fboxsep-2\fboxrule\relax}{\small
\textbf{Corrected version (verified by Claude).} This is a corrected version of the paper A.~Fern\'andez Ouaridi, I.~Kaygorodov, and C.~Mart\'in Gonz\'alez, \emph{Conservative algebras of $2$-dimensional algebras, IV}, Journal of Algebra and its Applications, 24 (2025), no.~6, 2550143, \doi{10.1142/S0219498825501439}.}}

\bigskip
\bigskip
\begin{center}
{\Large Conservative algebras of $2$-dimensional algebras, IV}

\bigskip
\bigskip
{\bf Amir Fern\'andez Ouaridi, Ivan Kaygorodov, and C\'andido Mart\'in Gonz\'alez}
\end{center}

\bigskip\bigskip

\noindent{\small\textbf{Abstract.} The notion of conservative algebras appeared in a paper by Kantor in 1972. Later, he defined the conservative algebra $W(n)$ of all algebras (i.e. bilinear maps) on the $n$-dimensional vector space. If $n>1$, then the algebra $W(n)$ does not belong to any well-known class of algebras (such as associative, Lie, Jordan, or Leibniz algebras). It looks like $W(n)$ plays a role in the theory of conservative algebras similar to the role of $\mathfrak{gl}_n$ in the theory of Lie algebras. Namely, an arbitrary conservative algebra can be obtained from a universal algebra $W(n)$ for some $n \in \mathbb{N}.$ The present paper is part of a series of papers dedicated to the study of the algebra $W(2)$ and its principal subalgebras.}



\bigskip
\bigskip
\tableofcontents

\bigskip

\newpage 
\section*{Introduction}
\addcontentsline{toc}{section}{Introduction}

A multiplication on a vector space $W$ is a bilinear mapping $W\times W \to W$. We denote by $(W,P)$ the algebra with underlying space $W$ and multiplication $P$. Given a vector space $W$, a linear mapping ${\bf A}:W\rightarrow W$, and a bilinear mapping ${\bf B}:W\times W\to W$, we can define a multiplication $[ {\bf A},{\bf B} ] :W\times W\to W$ by the formula
$$[ {\bf A}, {\bf B} ] (x,y)={\bf A}({\bf B}(x,y))-{\bf B}({\bf A}(x),y)-{\bf B}(x,{\bf A}(y))$$
for $x,y\in W$. For an algebra ${\bf A}$ with a multiplication $P$ and $x\in {\bf A}$ we denote by $L_x^P$ the operator of left multiplication by $x$. If the multiplication $P$ is fixed, we write $L_x$ instead of $L_x^P$.

In 1990 Kantor~\cite{Kantor90} defined the multiplication $\cdot$ on the set of all algebras (i.e. all multiplications) on the $n$-dimensional vector space $V_n$ as follows:
$${\bf A}\cdot {\bf B} = [L_{e}^{\bf A},{\bf B}],$$
where ${\bf A}$ and ${\bf B}$ are multiplications and $e\in V_n$ is some fixed vector.
Let $W(n)$  denote the algebra of all algebra structures on $V_n$ with multiplication defined above.
If $n>1$, then the algebra $W(n)$ does not belong to any well-known class of algebras (such as associative, Lie, Jordan, or Leibniz algebras). The algebra $W(n)$ turns out to be a conservative algebra (see below).

In 1972 Kantor~\cite{Kantor72} introduced conservative algebras as a generalization of Jordan algebras
 (also, see surveys about the study and recent results of conservative algebras and superalgebras \cite{k23,pp20}). Namely, an algebra ${\bf A}=(W,P)$ is called a {\it conservative algebra} if there is a new multiplication $F:W\times W\rightarrow W$ such that
\begin{eqnarray}\label{tojd_oper}
[L_b^P, [L_a^P,P]]=-[L_{F(a,b)}^P,P]
\end{eqnarray}
for all $a,b\in W$. In other words, the following identity holds for all $a,b,x,y\in W$:
\begin{multline}\label{tojdestvo_glavnoe}
b(a(xy)-(ax)y-x(ay))-a((bx)y)+(a(bx))y+(bx)(ay)\\
-a(x(by))+(ax)(by)+x(a(by))
=-F(a,b)(xy)+(F(a,b)x)y+x(F(a,b)y).
\end{multline}
The algebra $(W,F)$ is called an algebra {\it associated} to ${\bf A}$.
The main subclass of conservative algebras is the variety of terminal algebras,
which are defined by the identity (\ref{tojdestvo_glavnoe}) with
$F(a,b)=\frac{1}{3}(2 ab+ ba).$
It includes the varieties of Leibniz and Jordan algebras as subvarieties.

Let us recall some well-known results about conservative algebras. In~\cite{Kantor72} Kantor classified all simple conservative algebras and triple systems of second-order and defined the class of terminal algebras as algebras satisfying a certain identity. He proved that every terminal algebra is a conservative algebra and classified all simple finite-dimensional terminal algebras with left quasi-unit over an algebraically closed field of characteristic zero~\cite{Kantor89term}. Terminal trilinear operations were studied in~\cite{Kantor89tril}. After that, Cantarini and Kac classified simple finite-dimensional (and linearly compact) super-commutative and super-anticommutative conservative superalgebras and some generalization of these algebras (also known as ``rigid'' or quasi-conservative superalgebras) over an algebraically closed field of characteristic zero \cite{kac_can_10}.
The classification of all $2$-dimensional conservative and rigid (in the sense of Kac-Cantarini) algebras is given in \cite{cfk};
and also, the algebraic and geometric classification of nilpotent low dimensional terminal algebras is given in \cite{kkp19,kks19}.

The algebra $W(n)$ plays a similar role in the theory of conservative algebras as the Lie algebra of all $n\times n$ matrices $\mathfrak{gl}_n$ plays in the theory of Lie algebras. Namely, in~\cite{Kantor90}  Kantor considered the category $\mathcal{S}_n$ whose objects are conservative algebras of non-Jacobi dimension $n$. It was proven that the algebra $W(n)$ is the universal attracting object in this category, i.e., for every $M\in\mathcal{S}_n$ there exists a canonical homomorphism from $M$ into the algebra $W(n)$. In particular, all Jordan algebras of dimension $n$ with unity are contained in the algebra
$W(n)$. The same statement also holds for all noncommutative Jordan algebras of dimension $n$ with unity.
Some properties of the product in the algebra $W(n)$ were studied in \cite{kantorII,kay}.
The universal conservative superalgebra was constructed in \cite{kpp19}.
The study of low dimensional conservative algebras was started in \cite{kaylopo}.
The study of properties of $2$-dimensional algebras is also a  popular topic in non-associative algebras (see, for example, \cite{Giambruno,kayvo2,mp04,ps04,GR11,GR12}) and as we can see the study of properties of the algebra $W(2)$ could give some applications to the theory of $2$-dimensional algebras. So, an algebraic classification of all $2$-dimensional algebras with a left quasi-unit was obtained from the description of idempotents of the algebra $W(2)$~\cite{kayvo}.
Derivations and subalgebras of codimension 1 of the algebra $W(2)$ and of its principal subalgebras $W_2$ and $S_2$ were described \cite{kaylopo}.
Later, the automorphisms, one-sided ideals, idempotents,  local (and $2$-local) derivations, and automorphisms of
$W(2)$ and its principal subalgebras were described in \cite{arzikulov,kayvo}.
Note that $W_2$ and $S_2$ are simple terminal algebras with left quasi-unit from the classification of Kantor \cite{Kantor89term}.
The present paper is devoted to continuing the study of properties of $W(2)$ and its principal subalgebras. Throughout this paper, unless stated otherwise, $\mathbb{F}$ denotes a field of characteristic zero. All algebras are defined over $\mathbb{F}$.

The multiplication table of $W(2)$ is given by the following table.

\begin{center}
\begin{longtable}{c|c|c|c|c|c|c|c|c|}
      & $e_1$ & $e_2$ & $e_3$ & $e_4$ & $e_5$ & $e_6$ & $e_7$ & $e_8$ \\ \hline
$e_1$ & $-e_1$ & $-3e_2$ & $e_3$ & $3e_4$ & $-e_5$ & $e_6$ & $e_7$ & $-e_8$ \\ \hline
$e_2$ & $3e_2$ & $0$ & $2e_1$ & $e_3$ & $0$ & $-e_5$ & $e_8$ & $0$ \\ \hline
$e_3$ & $-2e_3$ & $-e_1$ & $-3e_4$ & $0$ & $e_6$ & $0$ & $0$ & $-e_7$ \\ \hline
$e_4$ & $0$ & $0$ & $0$ & $0$ & $0$ & $0$ & $0$ & $0$ \\ \hline
$e_5$ & $-2e_1$ & $-3e_2$ & $-e_3$ & $0$ & $-2e_5$ & $-e_6$ & $-e_7$ & $-2e_8$ \\ \hline
$e_6$ & $2e_3$ & $e_1$ & $3e_4$ & $0$ & $-e_6$ & $0$ & $0$ & $e_7$ \\ \hline
$e_7$ & $2e_3$ & $e_1$ & $3e_4$ & $0$ & $-e_6$ & $0$ & $0$ & $e_7$ \\ \hline
$e_8$ & $0$ & $e_2$ & $-e_3$ & $-2e_4$ & $0$ & $-e_6$ & $-e_7$ & $0$ \\ \hline
\end{longtable}
\end{center}

 \section{In\"on\"u-Wigner contractions of $W(2)$ and its subalgebras}

 The class of conservative algebras includes the variety of terminal algebras,
 which includes all Leibniz and Jordan algebras.
 On the other hand, the variety of terminal algebras is "dual" to the variety of commutative algebras
 (in the sense of generalized TTK-functor).
 The algebra $W(2)$ is not terminal, but its principal subalgebras $W_2$ and $S_2$ are terminal.
 Our main aim is to try to understand how "far"  the algebra $W(2)$ is from being a terminal algebra.
 For a particular answer  to our question, we will consider contractions to some certain subalgebras of $W(2)$
 and study its relations with the variety of terminal algebras.
 The {\it standard In\"on\"u-Wigner contraction} was introduced in \cite{IW}
 (the notion of In\"on\"u-Wigner contraction is well-related to degenerations, about degenerations see \cite{lisa} and references therein).
 We will call it the {\it IW contraction} for short.
 \begin{definition}
    Let $\mu,\chi$ represent algebras ${\bf A}$ and ${\bf B}$ respectively defined on a vector space $V$. Suppose that there are some elements $E_i^t\in V$ $(1\le i\le n$, $t\in  {\mathbb F}^*)$ such that $E^t=(E_1^t,\dots,E_n^t)$ is a basis of $V$ for any $t\in{\mathbb F}^*$ and the structure constants of $\mu$ in this basis are $\mu_{i,j}^k(t)$ for some polynomials $\mu_{i,j}^k(t)\in{\mathbb F}[t]$. If $\mu_{i,j}^k(0)=\chi_{i,j}^k$ for all $1\le i,j,k\le n$, then ${\bf A}\to {\bf B}$. To emphasize that the {\it parametrized basis} $E^t=(E_1^t,\dots,E_n^t)$ $(t\in{\mathbb F}^*)$ gives a degeneration between the algebras represented by the structures $\mu$ and $\chi$, we will write $\mu\xrightarrow{E^t}\chi$.
  Suppose that ${\bf A}_0$ is an $(n-m)$-dimensional subalgebra of the $n$-dimensional algebra ${\bf A}$ and $\mu$ is a structure representing ${\bf A}$ such that ${\bf A}_0$ corresponds to the subspace $\langle e_{m+1},\dots,e_n\rangle$ of $V$.
Then $\mu\xrightarrow{(te_1,\dots,te_m,e_{m+1},\dots,e_n)}\chi$ for some $\chi$ and the algebra ${\bf B}$ represented by $\chi$ is called the IW contraction of ${\bf A}$ with respect to ${\bf A}_0$.
\end{definition}

\subsection{IW contraction of $W(2)$}

\subsubsection{The algebra $\overline{W(2)}$}
The description of all subalgebras of codimension $1$ for the algebra $W(2)$ is given in \cite{kaylopo}.
Namely, $W(2)$ has only one $7$-dimensional subalgebra. It is generated by elements $e_1, e_3, e_4, e_5,e_6, e_7, e_8,$ and it is terminal. Let us consider the IW contraction $W(2) \xrightarrow{(e_1, t e_2, e_3, e_4, e_5, e_6,e_7,e_8)} \overline{W(2)}.$ It is easy to see, that the multiplication table of $\overline{W(2)}$ is given by the following table.

\begin{center}
\begin{longtable}{c|c|c|c|c|c|c|c|c|}
      & $e_1$ & $e_2$ & $e_3$ & $e_4$ & $e_5$ & $e_6$ & $e_7$ & $e_8$ \\ \hline
$e_1$ & $-e_1$ & $-3e_2$ & $e_3$ & $3e_4$ & $-e_5$ & $e_6$ & $e_7$ & $-e_8$ \\ \hline
$e_2$ & $3e_2$ & $0$ & $0$ & $0$ & $0$ & $0$ & $0$ & $0$ \\ \hline
$e_3$ & $-2e_3$ & $0$ & $-3e_4$ & $0$ & $e_6$ & $0$ & $0$ & $-e_7$ \\ \hline
$e_4$ & $0$ & $0$ & $0$ & $0$ & $0$ & $0$ & $0$ & $0$ \\ \hline
$e_5$ & $-2e_1$ & $-3e_2$ & $-e_3$ & $0$ & $-2e_5$ & $-e_6$ & $-e_7$ & $-2e_8$ \\ \hline
$e_6$ & $2e_3$ & $0$ & $3e_4$ & $0$ & $-e_6$ & $0$ & $0$ & $e_7$ \\ \hline
$e_7$ & $2e_3$ & $0$ & $3e_4$ & $0$ & $-e_6$ & $0$ & $0$ & $e_7$ \\ \hline
$e_8$ & $0$ & $e_2$ & $-e_3$ & $-2e_4$ & $0$ & $-e_6$ & $-e_7$ & $0$ \\ \hline
\end{longtable}
\end{center}
After a careful checking of the dimension of the algebra of derivations of $\overline{W(2)},$ we have
$\dim \mathfrak{Der} (\overline{W(2)})=3.$
Since
$\dim \mathfrak{Der} (W(2))=2,$
it follows that the degeneration
$W(2) \to \overline{W(2)}$ is primary, that is, there is no algebra ${\bf A}$ such that ${W(2)}\to {\bf A}$ and ${\bf A}\to \overline{W(2)}$, where ${\bf A}$ is neither isomorphic to $W(2)$ or $\overline{W(2)}$ (see \cite{degs22}).

\begin{Lem}
The algebra $\overline{W(2)}$ is a non-terminal conservative non-simple algebra.

\end{Lem}

\begin{proof}
The subspace
$\langle e_2,   e_3, e_4, e_6, e_7, e_8 \rangle$ gives a $6$-dimensional     ideal, showing that  $\overline{W(2)}$ is non-simple.
The non-terminal property follows from the direct verification of the terminal identity (for example, using a modification of the Wolfram code presented in \cite{kadyrov}).
The conservative property follows from the direct verification of the conservative identity with the additional multiplication $*:$
\begin{longtable}{lllll}

$e_1*e_1=-e_1$ &   $e_1*e_2=-e_2$&  $e_1*e_5=-2 e_1$ & $e_2*e_1=e_2$ & $e_2*e_5=-e_2$  \\
 $e_2*e_8=-e_2$ & $e_5*e_1=-2 e_1$ & $e_5*e_2=-2 e_2$  & $e_5*e_5=-4 e_1.$& \\

\end{longtable}
\end{proof}

\begin{Lem}\label{lem_codim1}
Let $\mathcal S$ be a subalgebra of $\overline{W(2)}$ of codimension $1$.
Then $\mathcal S$ is one of the following conservative subalgebras
\begin{longtable}{l}
${\mathcal S}_1 =\langle e_2, e_3, e_4, e_5, e_6, e_7, e_8 \rangle$, \\
${\mathcal S}_2 =\langle e_1, e_3, e_4, e_5, e_6, e_7, e_8 \rangle$, \\
${\mathcal S}_5^{\gamma} =\langle e_1-\gamma e_5, e_2, e_3, e_4, e_6, e_7, e_8 \rangle_{\gamma \in {\mathbb F}}$, \\
${\mathcal S}_7 =\langle e_1, e_2, e_3+e_6, e_3+e_7, e_4, e_5, e_8 \rangle$, \\
${\mathcal S}_{\alpha, \beta} =\langle e_1 +\alpha e_8, e_2, e_3, e_4, e_5+\beta e_8, e_6, e_7 \rangle_{\alpha, \beta \in {\mathbb F}}$, \\
${\mathcal S}^{s} =\langle e_1+2e_8, e_2, e_3+se_8, e_4, e_5+e_8, e_6-se_8, e_7-se_8 \rangle_{s \in {\mathbb F}^*}$,
\end{longtable}
where ${\mathcal S}_5={\mathcal S}_5^{0}$, and only ${\mathcal S}_1,$ ${\mathcal S}_2$, ${\mathcal S}_{0,0} $ and ${\mathcal S}_{-1,1} $
are terminal.

\end{Lem}

\begin{proof}
Let ${\mathcal S}=\ker f$, where $f=\sum_{i=1}^8 f_ie_i^*$ is a nonzero linear functional, and let $k$ be the largest index such that $f_k\neq0$; we may assume that $f_k=1$. Then ${\mathcal S}$ is spanned by the elements $e_i-f_ie_k$ $(i<k)$ and $e_j$ $(j>k)$, and ${\mathcal S}$ is a subalgebra if and only if $f(xy)=0$ for all these basis elements $x,y$. This gives a system of quadratic equations in $f_1,\ldots,f_{k-1}$. Solving these systems for $k=1,\ldots,8$ (we have checked the computations with a computer), we obtain that there are no solutions for $k=3,4,6$, and the solutions for the remaining values of $k$ are the following:
\begin{itemize}
\item $k=1$: $f=e_1^*$, which gives ${\mathcal S}_1$;
\item $k=2$: $f=e_2^*$, which gives ${\mathcal S}_2$;
\item $k=5$: $f=e_5^*+\gamma e_1^*$, which gives ${\mathcal S}_5^{\gamma}$;
\item $k=7$: $f=e_7^*+e_6^*-e_3^*$, which gives ${\mathcal S}_7$;
\item $k=8$: $f=e_8^*-\alpha e_1^*-\beta e_5^*$, which gives ${\mathcal S}_{\alpha,\beta}$, or $f=e_8^*-2e_1^*-e_5^*+s(e_6^*+e_7^*-e_3^*)$ with $s\neq0$, which gives ${\mathcal S}^{s}$.
\end{itemize}
The conservative property of the subalgebras ${\mathcal S}_5^{\gamma}$ $(\gamma\neq0)$, ${\mathcal S}_7$ and ${\mathcal S}^{s}$, as well as the fact that they are not terminal, also follow from a direct (computer-assisted) verification.

The conservative property of the subalgebra ${\mathcal S}_5$ is following from the direct verification of the conservative identity with the additional multiplication $*:$
\begin{longtable}{llll}
$e_1*e_1=-e_1$ &   $e_1*e_2=-e_2$&  $e_2*e_1= e_2$  & $e_2*e_8=- e_2$
\end{longtable}

Let us give the multiplication table of $\overline{W(2)}$ in more useful way  (here the subalgebra $\langle e_1, \ldots, e_7 \rangle$ gives ${\mathcal S}_{\alpha, \beta}$):

\begin{longtable}{llll}
$e_1e_1=-e_1$&$e_1e_2=(-3+\alpha) e_2$&$e_1e_3=(1-\alpha) e_3$&$e_1e_4=(3-2 \alpha) e_4$\\

$e_1e_5=-e_5$&$e_1e_6=(1-\alpha) e_6$&$e_1e_7=(1-\alpha) e_7$&$e_1e_8=-e_8$\\

$e_2e_1=3 e_2$&$e_3e_1=-2 e_3-\alpha e_7$&$e_3e_3=-3 e_4$&$e_3e_5=e_6-\beta e_7$\\

$e_3e_8=-e_7$&$e_5e_1=-2 e_1$&$e_5e_2=(-3+\beta) e_2$&$e_5e_3=(-1-\beta) e_3$\\

$e_5e_4=-2 \beta e_4$&$e_5e_5=-2 e_5$&$e_5e_6=(-1-\beta) e_6$&$e_5e_7=(-1-\beta) e_7$\\

$e_5e_8=-2 e_8$&$e_6e_1=2 e_3+\alpha e_7$&$e_6e_3=3 e_4$&$e_6e_5=-e_6+\beta e_7$\\

$e_6e_8=e_7$&$e_7e_1=2 e_3+\alpha e_7$&$e_7e_3=3 e_4$&$e_7e_5=-e_6+\beta e_7$\\

$e_7e_8=e_7$&$e_8e_2=e_2$&$e_8e_3=-e_3$&$e_8e_4=-2 e_4$\\

& $e_8e_6=-e_6$&$e_8e_7=-e_7$
\end{longtable}

The conservative property of the subalgebra ${\mathcal S}_{\alpha, \beta}$ is following from the direct verification of the conservative identity with the additional multiplication $*:$
\begin{longtable}{llll}
$e_1*e_1=-e_1$&$
e_1*e_2=-e_2$&$
e_1*e_5=-2 e_1$&$
e_2*e_1=(1-\alpha) e_2$\\$

e_2*e_5=(-1-\beta) e_2$&$
e_5*e_1=-2 e_1$&$
e_5*e_2=-2 e_2$&$
e_5*e_5=-4 e_1$
\end{longtable}

\end{proof}

\subsubsection{Algebras $\overline{{\mathcal S}}$}
In the present subsection, we have to talk about contractions of the algebra $\overline{W(2)}$ to its subalgebra of codimension $1.$

$\bullet$ $\overline{W(2)} \xrightarrow{(te_1, {\mathcal S}_1)} \overline{{\mathcal S}_1}.$
It is easy to see that the multiplication of $\overline{{\mathcal S}_1}$ is given by the following table.

{\footnotesize
\begin{longtable}{lllllll}

$e_3e_3=-3 e_4$&$
e_3e_5=e_6$&$
e_3e_8=-e_7$&$
e_5e_1=-2 e_1$&$
e_5e_2=-3 e_2$&$
e_5e_3=-e_3$&$
e_5e_5=-2 e_5$\\

$e_5e_6=-e_6$&$
e_5e_7=-e_7$&$
e_5e_8=-2 e_8$&$
e_6e_3=3 e_4$&$
e_6e_5=-e_6$&$
e_6e_8=e_7$&$
e_7e_3=3 e_4$\\

$e_7e_5=-e_6$&$
e_7e_8=e_7$&$
e_8e_2=e_2$&$
e_8e_3=-e_3$&$
e_8e_4=-2 e_4$&$
e_8e_6=-e_6$&$
e_8e_7=-e_7$

\end{longtable}}

\begin{Lem}
The algebra $\overline{{\mathcal S}_1}$ is terminal.
\end{Lem}

 $\bullet$ $\overline{W(2)} \xrightarrow{(te_5, {\mathcal S}_5)} \overline{{\mathcal S}_5}.$
It is easy to see that the multiplication of $\overline{{\mathcal S}_5}$ is given by the following table.

\begin{longtable}{llllll}
$e_1e_1=-e_1$&$
e_1e_2=-3 e_2$&$
e_1e_3=e_3$&$
e_1e_4=3 e_4$&$
e_1e_5=-e_5$&$
e_1e_6=e_6$\\$
e_1e_7=e_7$&$
e_1e_8=-e_8$&$
e_2e_1=3 e_2$&$
e_3e_1=-2 e_3$&$
e_3e_3=-3 e_4$&$
e_3e_8=-e_7$\\$

e_6e_1=2 e_3$&$
e_6e_3=3 e_4$&$
e_6e_8=e_7$&$
e_7e_1=2 e_3$&$
e_7e_3=3 e_4$&$
e_7e_8=e_7$\\$

e_8e_2=e_2$&$
e_8e_3=-e_3$&$
e_8e_4=-2 e_4$&$
e_8e_6=-e_6$&$
e_8e_7=-e_7$
\end{longtable}

\begin{Lem}
The algebra $\overline{{\mathcal S}_5}$ is a non-terminal conservative algebra.
\end{Lem}

 \begin{proof}
 The conservative property of the algebra  $\overline{{\mathcal S}_5}$  is following from the direct verification of the conservative identity with the additional multiplication $*:$
\begin{longtable}{llll}
$e_1*e_1=-e_1$&$
e_1*e_2=-e_2$&$
e_2*e_1=e_2$&$
e_2*e_8=-e_2$
\end{longtable}

\end{proof}

$\bullet$ $\overline{W(2)} \xrightarrow{(te_8, {\mathcal S}_{\alpha,\beta})} \overline{{\mathcal S}_{\alpha,\beta}}.$
It is easy to see that the multiplication of $\overline{{\mathcal S}_{\alpha,\beta}}$ is given by the following table.

\begin{longtable}{llll}

$e_1e_1=-e_1$&$
e_1e_2=(-3+\alpha) e_2$&$
e_1e_3=(1-\alpha) e_3$&$
e_1e_4=(3-2 \alpha) e_4$\\$

e_1e_5=-e_5$&$
e_1e_6=(1-\alpha) e_6$&$
e_1e_7=(1-\alpha) e_7$&$
e_1e_8=-e_8$\\$

e_2e_1=3 e_2$&$
e_3e_1=-2 e_3-\alpha e_7$&$
e_3e_3=-3 e_4$&$
e_3e_5=e_6-\beta e_7$\\$

e_5e_1=-2 e_1$&$
e_5e_2=(-3+\beta) e_2$&$
e_5e_3=(-1-\beta) e_3$&$
e_5e_4=-2 \beta e_4$\\$

e_5e_5=-2 e_5$&$
e_5e_6=(-1-\beta) e_6$&$
e_5e_7=(-1-\beta) e_7$&$
e_5e_8=-2 e_8$\\$

e_6e_1=2 e_3+\alpha e_7$&$
e_6e_3=3 e_4$&$
e_6e_5=-e_6+\beta e_7$&$
e_7e_1=2 e_3+\alpha e_7$\\&$

e_7e_3=3 e_4$&$
e_7e_5=-e_6+\beta e_7$

\end{longtable}

\begin{Lem}
The algebra $\overline{{\mathcal S}_{\alpha,\beta}}$ is a conservative algebra;
and it is a terminal algebra if and only if $(\alpha,\beta)=(-1,1)$ or $(\alpha,\beta)=(0,0)$.
\end{Lem}

 \begin{proof}
 The conservative property of the algebra  $\overline{{\mathcal S}_{\alpha,\beta}}$  is following from the direct verification of the conservative identity with the additional multiplication $*:$
\begin{longtable}{llll}
$e_1*e_1=-e_1$&$
e_1*e_2=-e_2$&$
e_1*e_5=-2 e_1$&$
e_2*e_1=(1-\alpha) e_2$\\$

e_2*e_5=(-1-\beta) e_2$&$
e_5*e_1=-2 e_1$&$
e_5*e_2=-2 e_2$&$
e_5*e_5=-4 e_1$
\end{longtable}
\end{proof}

\subsubsection{The algebra $\widehat{W(2)}$}
The second interesting "big" subalgebra of $W(2)$ is $W_2,$ which is generated by $e_1, \ldots, e_6.$
 Let us consider the IW contraction $W(2) \xrightarrow{(e_1,  e_2, e_3, e_4, e_5, e_6, te_7,te_8)} {\widehat{W(2)}}.$
It is easy to see, that the multiplication table (for nonzero products) of ${\widehat{W(2)}}$ is given by the following table.

\begin{center}
\begin{longtable}{c|c|c|c|c|c|c|c|c|}
      & $e_1$ & $e_2$ & $e_3$ & $e_4$ & $e_5$ & $e_6$ & $e_7$ & $e_8$ \\ \hline
$e_1$ & $-e_1$ & $-3e_2$ & $e_3$ & $3e_4$ & $-e_5$ & $e_6$ & $e_7$ & $-e_8$ \\ \hline
$e_2$ & $3e_2$ & $0$ & $2e_1$ & $e_3$ & $0$ & $-e_5$ & $e_8$ & $0$ \\ \hline
$e_3$ & $-2e_3$ & $-e_1$ & $-3e_4$ & $0$ & $e_6$ & $0$ & $0$ & $-e_7$ \\ \hline
$e_4$ & $0$ & $0$ & $0$ & $0$ & $0$ & $0$ & $0$ & $0$ \\ \hline
$e_5$ & $-2e_1$ & $-3e_2$ & $-e_3$ & $0$ & $-2e_5$ & $-e_6$ & $-e_7$ & $-2e_8$ \\ \hline
$e_6$ & $2e_3$ & $e_1$ & $3e_4$ & $0$ & $-e_6$ & $0$ & $0$ & $e_7$ \\ \hline
\end{longtable}
\end{center}

\begin{Lem}
The algebra ${\widehat{W(2)}}$ is terminal.
\end{Lem}

\subsubsection{The algebra $\widehat{\widehat{W(2)}}$}
The next interesting  subalgebra of $W(2)$ is $S_2,$ which is generated by $e_1, \ldots, e_4.$
 Let us consider the IW contraction $W(2) \xrightarrow{(e_1,  e_2, e_3, e_4, te_5, te_6, te_7,te_8)} {\widehat{\widehat{W(2)}}}.$
It is easy to see, that the multiplication table (for nonzero products) of ${\widehat{\widehat{W(2)}}}$ is given by the following table.

\begin{center}
\begin{longtable}{c|c|c|c|c|c|c|c|c|}
      & $e_1$ & $e_2$ & $e_3$ & $e_4$ & $e_5$ & $e_6$ & $e_7$ & $e_8$ \\ \hline
$e_1$ & $-e_1$ & $-3e_2$ & $e_3$ & $3e_4$ & $-e_5$ & $e_6$ & $e_7$ & $-e_8$ \\ \hline
$e_2$ & $3e_2$ & $0$ & $2e_1$ & $e_3$ & $0$ & $-e_5$ & $e_8$ & $0$ \\ \hline
$e_3$ & $-2e_3$ & $-e_1$ & $-3e_4$ & $0$ & $e_6$ & $0$ & $0$ & $-e_7$ \\ \hline
\end{longtable}
\end{center}

\begin{corollary}
The algebra ${\widehat{\widehat{W(2)}}}$ is terminal.
\end{corollary}

\subsubsection{The algebra $ \widetilde{W(2)}$}
The next interesting  subalgebra of $W(2)$ is  generated by $e_1$ and $e_2.$
 Let us consider the IW contraction $W(2) \xrightarrow{(e_1,  e_2, te_3, te_4, te_5, te_6, te_7,te_8)}  {\widetilde{W(2)} }.$
It is easy to see, that the multiplication table (for nonzero products) of $ {\widetilde{W(2)} }$ is given by the following table.

\begin{center}
\begin{longtable}{c|c|c|c|c|c|c|c|c|}
      & $e_1$ & $e_2$ & $e_3$ & $e_4$ & $e_5$ & $e_6$ & $e_7$ & $e_8$ \\ \hline
$e_1$ & $-e_1$ & $-3e_2$ & $e_3$ & $3e_4$ & $-e_5$ & $e_6$ & $e_7$ & $-e_8$ \\ \hline
$e_2$ & $3e_2$ & $0$ & $0$ & $e_3$ & $0$ & $-e_5$ & $e_8$ & $0$ \\ \hline
$e_3$ & $-2e_3$ & $0$ & $0$ & $0$ & $0$ & $0$ & $0$ & $0$ \\ \hline
$e_6$ & $2e_3$ & $0$ & $0$ & $0$ & $0$ & $0$ & $0$ & $0$ \\ \hline
$e_7$ & $2e_3$ & $0$ & $0$ & $0$ & $0$ & $0$ & $0$ & $0$ \\ \hline
\end{longtable}
\end{center}

\begin{Lem}
The algebra $ {\widetilde{W(2)} }$ is a non-Leibniz, non-Jordan terminal algebra.
\end{Lem}

\subsubsection{The algebra $ \widetilde{\widetilde{W(2)}}$}
The last interesting subalgebra of $W(2)$ is generated by $e_1.$
Let us consider the IW contraction $W(2) \xrightarrow{(e_1,  t e_2, te_3, te_4, te_5, te_6, te_7,te_8)} \widetilde{\widetilde{W(2)}}.$
It is easy to see, that the multiplication table (for nonzero products) of $\widetilde{\widetilde{W(2)}}$ is given by the following table.

\begin{center}
\begin{longtable}{c|c|c|c|c|c|c|c|c|}
      & $e_1$ & $e_2$ & $e_3$ & $e_4$ & $e_5$ & $e_6$ & $e_7$ & $e_8$ \\ \hline
$e_1$ & $-e_1$ & $-3e_2$ & $e_3$ & $3e_4$ & $-e_5$ & $e_6$ & $e_7$ & $-e_8$ \\ \hline
$e_2$ & $3e_2$ & $0$ & $0$ & $0$ & $0$ & $0$ & $0$ & $0$ \\ \hline
$e_3$ & $-2e_3$ & $0$ & $0$ & $0$ & $0$ & $0$ & $0$ & $0$ \\ \hline
$e_6$ & $2e_3$ & $0$ & $0$ & $0$ & $0$ & $0$ & $0$ & $0$ \\ \hline
$e_7$ & $2e_3$ & $0$ & $0$ & $0$ & $0$ & $0$ & $0$ & $0$ \\ \hline
\end{longtable}
\end{center}

\begin{Lem}
The algebra $\widetilde{\widetilde{W(2)}}$ is a non-Leibniz, non-Jordan terminal algebra.
\end{Lem}

\subsection{IW contraction of $S_2$ and $W_2$}
Thanks to \cite{kaylopo}, algebras $S_2$ and $W_2$ have also only one subalgebra of codimension $1,$
which are  $\langle e_1,   e_3, e_4 \rangle$ and $\langle e_1,  e_3, e_4, e_5, e_6 \rangle.$
Other important subalgebras of $S_2$ and $W_2$ are  $\langle e_1 \rangle,$ $\langle e_1, e_2 \rangle$ and
$\langle e_1, e_2, e_3, e_4 \rangle.$
All contractions of $S_2$ (and $W_2$) with respect to all cited subalgebras can be obtained
as $4$-dimensional subalgebras $\langle e_1, e_2, e_3, e_4 \rangle$ ($6$-dimensional subalgebras $\langle e_1, \ldots, e_6 \rangle$) of the following algebras
$\overline{W(2)},$  ${\widehat{\widehat{W(2)}}},$ $ {\widetilde{W(2)} }$ and  $\widetilde{\widetilde{W(2)}}.$
All these subalgebras are non-Leibniz, non-Jordan terminal algebras.

\section{Varieties related to $W(2)$ and  its subalgebras}

\subsection{Identities}
Let $\F$ be a field of characteristic zero and
$\F\langle x_1, x_2, \ldots, x_n\rangle$ the free nonassociative $\F$-algebra in $n$ indeterminates.
Let ${\bf A}$ be any  algebra  and
${\mathfrak S}^n_{\bf A}$ the subspace of $\F\langle x_1, x_2, \ldots, x_n\rangle$ of all $n$-linear $w(x_1,x_2,\ldots,x_n)$ which vanish on ${\bf A}$ (so $w(x_1,x_2,\ldots,x_n)$ is of degree one in each variable). We have studied by direct verification the subspace ${\mathfrak S}^n_{\bf A}$ for $n=3,4$ of $W(2)$ and of its IW contractions mentioned in this paper.

\begin{prop}\label{prop_id1}
In the following table, we summarize the dimension of the subspaces ${\mathfrak S}^n_{\bf A}$ for ${\bf A} \in \{
W(2), \overline{W(2)},
\overline{\mathcal{S}_1}, \overline{\mathcal{S}_5}, \overline{\mathcal{S}_{-1,1}}, \overline{\mathcal{S}_{0,0}},
\widehat{W(2)}, \widehat{\widehat{W(2)}},
{\widetilde{W(2)} }\}$ and $n=3,4$.

{\begin{longtable}{|c|c|c|c|c|}

\hline
Algebra $({\bf A})$  & ${\rm dim}({\mathfrak S}^3_{\bf A})$ &  ${\rm dim}({\mathfrak S}^4_{\bf A})$ & Comments\\ \hline

$W(2)$ &

0 &

0 &

non-terminal conservative \\ \hline

$\overline{W(2)}$ &

0 &

20 &

non-terminal, conservative \\ \hline

$\overline{\mathcal{S}_1}$ &

0 &

64 &

terminal \\ \hline

$\overline{\mathcal{S}_5}$ &

0 &

40 &

non-terminal, conservative \\ \hline

$\overline{\mathcal{S}_{-1, 1}}$ &

0 &

64 &

terminal\\ \hline

$\overline{\mathcal{S}_{0,0}}$ &

0 &

44 &

terminal \\ \hline

$\widehat{W(2)}$ &

0 &

24 &

terminal \\ \hline

$\widehat{\widehat{W(2)}}$ &

0 &

47 &

terminal \\ \hline

${\widetilde{W(2)}}$ &

0 &

82 &

terminal \\ \hline

${\widetilde{\widetilde{W(2)}}}$ &

2 &

101 &

terminal \\ \hline

\end{longtable}}
\noindent
 Moreover, if ${\bf A}=\overline{\mathcal{S}_{\alpha, \beta}}$ then $\dim({\mathfrak S}^3_{\bf A})=0$ for $(\alpha, \beta)\neq(2,1)$ and $(\alpha, \beta)\neq(0, -3)$.
\end{prop}

\begin{proof}

We have determined the spaces ${\mathfrak S}^n_{\bf A}$ for $n=3,4$ by constructing an arbitrary $n$-linear map $w(x_1,\ldots ,x_n)$ and solving $w(x_1,\ldots ,x_n)=0$ in ${\bf A}$ using Wolfram.

\end{proof}

\begin{prop}\label{prop_id2}
If ${\bf A}=\overline{\mathcal{S}_{2,1}}$ then $\dim({\mathfrak S}^3_{\bf A})=3$ and a basis of the $\F$-vector space ${\mathfrak S}^3_{\bf A}$ is the set of identities:
{\begin{multicols}{3}
\begin{enumerate}
\item $x_{1}(x_{2} x_{3}) - x_{2} (x_{1} x_{3}),$

\item $x_{2} (x_{3} x_{1}) - x_{3} (x_{2} x_{1}),$

\item $x_{3} (x_{1} x_{2}) - x_{1} (x_{3} x_{2}) .$
\end{enumerate}
\end{multicols}}
\end{prop}

Now, consider the following identities
\footnote{Let us note that the present type of identities plays a special role in the theory of right-symmetric algebras (see, \cite{askar,askar2}).}:
\[
\mathfrak{st}^n_1=\sum\limits_{\sigma \in \mathbb S_n} (-1)^{\sigma} (\ldots(x_{\sigma(1)}x_{\sigma(2)})x_{\sigma(3)}\ldots) x_{\sigma(n)}
\]
\mbox{and}
\[
\mathfrak{st}^n_2=\sum\limits_{\sigma \in \mathbb S_n} (-1)^{\sigma} x_{\sigma(n)}(\ldots x_{\sigma(3)}(x_{\sigma(2)}x_{\sigma(1)})\ldots).
\]
   It is clear that if an algebra satisfies $\mathfrak{st}^n_1$ (resp. $\mathfrak{st}^n_2$); 
   it also satisfies $\mathfrak{st}^{n+1}_1$ (resp. $\mathfrak{st}^{n+1}_2$).

\begin{prop}\label{prop_id3}
If ${\bf A}=\overline{\mathcal{S}_{0,-3}}$ then $\dim({\mathfrak S}^3_{\bf A})=1$. This $\F$-vector space is generated by $$2 \mathfrak{st}^3_1 +3   \mathfrak{st}^3_2.$$

\end{prop}

\begin{prop}\label{prop_id4}
If ${\bf A}=\widetilde{\widetilde{W(2)}}$ then $\dim({\mathfrak S}^3_{\bf A})=2$ and a basis of the $\F$-vector space ${\mathfrak S}^3_{\bf A}$ is the set of identities $\mathfrak{st}^3_1, \mathfrak{st}^3_2$.

\end{prop}

We have studied the space ${\mathfrak S}^n_{\bf A}$ for the subalgebras of $W(2)$ mentioned in this paper. We have also studied the identities $\mathfrak{st}^n_1$ and $\mathfrak{st}^n_2$ for these subalgebras for $n=3,4,5$.

\begin{prop}\label{prop_id5}
In the following table, we summarize the dimension of the subspaces ${\mathfrak S}^n_{\bf A}$ for ${\bf A}$ a subalgebra of $W(2)$ and $n=3,4$.

\begin{longtable}{|l|c|c|l|}

\hline
Subalgebra $({\bf A})$  & ${\rm dim}({\mathfrak S}^3_{\bf A})$ &  ${\rm dim}({\mathfrak S}^4_{\bf A})$ & Comments\\ \hline

$B_2:=\langle e_1, e_3, e_4, e_5, e_6, e_7, e_8\rangle$ &

0 &

64 &

  no $\mathfrak{st}^4_1$, no $\mathfrak{st}^4_2$, $\mathfrak{st}^5_1, \mathfrak{st}^5_2$ \\ \hline
$W_2=\langle e_1, e_2, e_3, e_4, e_5, e_6\rangle$ &

0 &

24 &

  no $\mathfrak{st}^5_1$, no $\mathfrak{st}^4_2,$  $\mathfrak{st}^5_2$ \\ \hline
$C_2:=\langle e_1, e_3, e_4, e_5, e_6\rangle$ &

0 &

64 &

no $\mathfrak{st}^4_1$, no $\mathfrak{st}^4_2$, $\mathfrak{st}^5_1,$ $\mathfrak{st}^5_2$ \\ \hline
$S_2=\langle e_1, e_2, e_3, e_4\rangle$ &

3 &

86 &

no $\mathfrak{st}^4_1$, no $\mathfrak{st}^4_2$, $\mathfrak{st}^5_1,$ $\mathfrak{st}^5_2$ \\ \hline
$D_2:=\langle e_1, e_3, e_4\rangle$ &

6 &

110 &

$\mathfrak{st}^3_1, \mathfrak{st}^3_2$ \\ \hline
$E_2:=\langle e_1, e_2\rangle$ &

8 &

115 &

 $\mathfrak{st}^3_1, \mathfrak{st}^3_2$\\ \hline

\end{longtable}

Moreover, ${\mathfrak S}^4_{ B_2}={\mathfrak S}^4_{ C_2};$
all present algebras are terminal,  non-Leibniz and non-Jordan.

\end{prop}

\begin{prop}\label{prop_id6}
If ${\bf A}=S_2$, the subalgebra of $W(2)$ generated by $e_1, \ldots, e_4$, then a basis of the $\F$-vector space ${\mathfrak S}^3_{\bf A}$ is the set of identities:
{ 
\begin{enumerate}
   \item $30 x_{1} (x_{2} x_{3})-42 x_{1}
   (x_{3} x_{2})-25 x_{2}
   (x_{1} x_{3})+7 x_{2} (x_{3}
   x_{1})+39 x_{3} (x_{1} x_{2})-9
   x_{3} (x_{2} x_{1})+10
   (x_{1} x_{3}) x_{2}+15
   (x_{2} x_{1}) x_{3}-19
   (x_{3} x_{1}) x_{2}-6 (x_{3}
   x_{2}) x_{1}$,
   \item $-5 x_{1} (x_{2} x_{3})+11 x_{1}
   (x_{3} x_{2})+5 x_{2} (x_{1}
   x_{3})-11 x_{2} (x_{3} x_{1})-12
   x_{3} (x_{1} x_{2})+12 x_{3}
   (x_{2} x_{1})-5 (x_{1}
   x_{3}) x_{2}+5 (x_{2} x_{3})
   x_{1}+2 (x_{3} x_{1}) x_{2}-2
   (x_{3} x_{2}) x_{1} $,
   \item $ -3 x_{1} (x_{2} x_{3}) - 3 x_{1} (x_{3} x_{2}) +
    4 x_{2} (x_{1} x_{3}) - 4 x_{2} (x_{3} x_{1}) +
    3 x_{3} (x_{1} x_{2}) + 3 x_{3} (x_{2} x_{1}) + 3 (x_{1} x_{2}) x_{3} +
    2 (x_{1} x_{3}) x_{2} - 2 (x_{3} x_{1}) x_{2} - 3 (x_{3} x_{2}) x_{1}$.

\end{enumerate}}

\end{prop}

Finally, we have studied the family of identities $\mathfrak{st}^n_1$ and $\mathfrak{st}^n_2$ for $W(2)$ and its contractions.

\begin{prop}
\label{stprop}
In the following table, we summarize which identities from the families $\mathfrak{st}^n_1$ and $\mathfrak{st}^n_1$ are satisfied for every contraction of $W(2)$, for $n=3,4,5$.

{\begin{longtable}{|c|c|c|c|c|c|c|c|}

\hline
Algebra & $\mathfrak{st}^3_1$ &  $\mathfrak{st}^4_1$ & $\mathfrak{st}^5_1$ & $\mathfrak{st}^3_2$ &  $\mathfrak{st}^4_2$ & $\mathfrak{st}^5_2$\\ \hline

$W(2)$ &

\xmark &

\xmark &

\xmark &

\xmark &

\xmark &

\cmark \\ \hline

$\overline{W(2)}$ &

\xmark &

\xmark &

\cmark &

\xmark &

\xmark &

\cmark \\ \hline

$\overline{\mathcal{S}_1}$ &

\xmark &

\xmark &

\cmark &

\xmark &

\xmark &

\cmark \\ \hline

$\overline{\mathcal{S}_5}$ &

\xmark &

\xmark &

\cmark &

\xmark &

\xmark &

\cmark \\ \hline

$\overline{\mathcal{S}_{\alpha,\beta}}$ &

\xmark &

$\cmark_{\alpha=\frac{3+\beta}{2}}$ &

\cmark &

$\cmark_{(\alpha, \beta)=(2, 1)}$ &

$\cmark_{\alpha=\frac{3+\beta}{2}}$ &

\cmark \\ \hline

$\widehat{W(2)}$ &

\xmark &

\xmark &

\xmark &

\xmark &

\xmark &

\cmark \\ \hline

$\widehat{\widehat{W(2)}}$ &

\xmark &

\xmark &

\cmark &

\xmark &

\xmark &

\cmark \\ \hline

${\widetilde{W(2)}}$ &

\xmark &

\cmark &

\cmark &

\xmark &

\cmark &

\cmark \\ \hline

\end{longtable}

}

\end{prop}

\begin{corollary}
${\mathfrak S}^4_{ W(n)}=0$ and
${\mathfrak S}^5_{ W(2)}\neq0$.

\end{corollary}
The present corollary gives the following question.

\begin{question}
Find mininal  $k,$ such that ${\mathfrak S}^k_{ W(n)}\neq0$.
Is $ W(n)$ satisfying $\mathfrak{st}^k_2$?
\end{question}

\subsection{Other degree five identities for $W(2)$}

In this subsection, we are interested in finding other degree five identities for $W(2)$. Consider the set of free monomials  $w(x_1,x_2, x_3, x_4,x_5)$ of degree five up to permutations of the variables. There are exactly fourteen monomials:
\begin{longtable}{llll}
$w_1(x_1,x_2,x_3,x_4,x_5)=(((x_{1}x_{2})x_{3})x_{4})x_{5}$&
$w_2(x_1,x_2,x_3,x_4,x_5)=((x_{1}x_{2})x_{3})(x_{4}x_{5})$\\
$w_3(x_1,x_2,x_3,x_4,x_5)=((x_{1}x_{2})(x_{3}x_{4}))x_{5}$&
$w_4(x_1,x_2,x_3,x_4,x_5)=(x_{1}x_{2})((x_{3}x_{4})x_{5})$\\
$w_5(x_1,x_2,x_3,x_4,x_5)=(x_{1}x_{2})(x_{3}(x_{4}x_{5}))$&
$w_6(x_1,x_2,x_3,x_4,x_5)=((x_{1}(x_{2}x_{3}))x_{4})x_{5}$\\
$w_7(x_1,x_2,x_3,x_4,x_5)=(x_{1}(x_{2}x_{3}))(x_{4}x_{5})$&
$w_8(x_1,x_2,x_3,x_4,x_5)=(x_{1}((x_{2}x_{3})x_{4}))x_{5}$\\
$w_9(x_1,x_2,x_3,x_4,x_5)=x_{1}(((x_{2}x_{3})x_{4})x_{5})$&
$w_{10}(x_1,x_2,x_3,x_4,x_5)=x_{1}((x_{2}x_{3})(x_{4}x_{5}))$\\
$w_{11}(x_1,x_2,x_3,x_4,x_5)=(x_{1}(x_{2}(x_{3}x_{4})))x_{5}$&
$w_{12}(x_1,x_2,x_3,x_4,x_5)=x_{1}((x_{2}(x_{3}x_{4}))x_{5})$\\
$w_{13}(x_1,x_2,x_3,x_4,x_5)=x_{1}(x_{2}((x_{3}x_{4})x_{5}))$&
$w_{14}(x_1,x_2,x_3,x_4,x_5)=x_{1}(x_{2}(x_{3}(x_{4}x_{5})))$\\
\end{longtable}
Now, consider the $\mathbb{F}$-vector spaces ${\mathfrak Z}^i$ generated by the set: $$\left\{w_i(x_{\sigma(1)},x_{\sigma(2)},x_{\sigma(3)},x_{\sigma(4)},x_{\sigma(5)}): \sigma \in \mathbb S_5 \right\}.$$
Denote by ${\mathfrak Z}^i_{\bf A}$ the subspace of ${\mathfrak Z}^i$ of  all $5$-linear polynomials vanishing on $\bf A$. Then we have the following result regarding the dimension of these subspaces:

\begin{prop}
If ${\bf A}=W(2)$, then ${\rm dim}({\mathfrak Z}^i_{\bf A})=0$ for $1\leq i \leq 13$ and ${\rm dim}({\mathfrak Z}^{14}_{\bf A})=5$. A basis of the $\F$-vector space ${\mathfrak Z}^{14}_{\bf A}$ is the set of identities:

{\begin{enumerate}
\item $\sum\limits_{\sigma \in \mathbb S_4} (-1)^{\sigma} x_{\sigma(2)}(x_{\sigma(3)}(x_{\sigma(4)}(x_{\sigma(5)}x_{1}))).$
\item $\sum\limits_{\sigma \in \mathbb S_4} (-1)^{\sigma} x_{\sigma(1)}(x_{\sigma(3)}(x_{\sigma(4)}(x_{\sigma(5)}x_{2}))).$
\item $\sum\limits_{\sigma \in \mathbb S_4} (-1)^{\sigma} x_{\sigma(1)}(x_{\sigma(2)}(x_{\sigma(4)}(x_{\sigma(5)}x_{3}))).$
\item $\sum\limits_{\sigma \in \mathbb S_4} (-1)^{\sigma} x_{\sigma(1)}(x_{\sigma(2)}(x_{\sigma(3)}(x_{\sigma(5)}x_{4}))).$
\item $\sum\limits_{\sigma \in \mathbb S_4} (-1)^{\sigma} x_{\sigma(1)}(x_{\sigma(2)}(x_{\sigma(3)}(x_{\sigma(4)}x_{5}))).$

\end{enumerate}}

Moreover, the linear combination with parameters $(1,-1,1,-1,1)$ is $\mathfrak{st}^5_2$.

\end{prop}

\subsection{Central extensions}
The notion of central extensions appeared in the study of Lie algebras,
but it can be considered in an arbitrary variety of algebras (see, for example, \cite{klp}).
The calculation of central extensions of an algebra ${\bf A}$ of dimension $n$ from a certain variety of algebras
gives the classification of all algebras with $(k-n)$-dimensional annihilator, such that its factor algebra by the annihilator is isomorphic to ${\bf A}$ (see, for example, \cite{antonio}). These calculations are carried out by studying the cohomology, with respect to a polynomial identity, of the algebra ${\bf A}$. In this section, we are interested in the central extensions of the contractions and subalgebras of $W(2)$ considered in this paper. Some of these contractions and subalgebras have turned out to be terminal, i.e., they satisfy the terminal identity (degree four). The following result is about these particular algebras.

\begin{prop}\label{zcontractions}
There are no terminal central extensions of the terminal contractions of $W(2)$: $\overline{\mathcal{S}_1}$, $\overline{\mathcal{S}_{-1,1}}$, $\overline{\mathcal{S}_{0,0}}$, $\widehat{W(2)}$, $\widehat{\widehat{W(2)}}$,  $\widetilde{W(2)}$ and $\widetilde{\widetilde{W(2)}}.$
\end{prop}
\begin{proof}
Recall that if $Z^{2}_P\left( {\bf A},{\mathbb{F}}\right)$ denotes the space of cocycles with respect to the polynomial identity $P$ of the algebra ${\bf A}$, $B^{2}\left( {\bf A},{\mathbb{F}}\right)$ denotes the space of coborders of the algebra ${\bf A}$ and $H^{2}_P\left( {\bf A},{\mathbb{F}}\right):=Z^{2}_P\left( {\bf A},{\mathbb{F}}\right)/B^{2}\left( {\bf A},{\mathbb{F}}\right)$ denotes the cohomology space with respect to the polynomial identity $P$ of the algebra ${\bf A}$, then if $H^{2}_P\left( {\bf A},{\mathbb{F}}\right)$ is trivial, we have that ${\bf A}$ has no central extensions for the identity $P$. Now, fix $P=T$ the terminal identity. Thus, the result is proven by direct calculation of the cohomology space, obtaining that $H^{2}_T\left( {\bf A},{\mathbb{F}}\right)$ is trivial for any of the terminal contractions ${\bf A}$ considered.
\end{proof}

\begin{prop}
There are no terminal central extensions of the terminal subalgebras of $W(2)$: $B_2, W_2, C_2, S_2, D_2, E_2$.
\end{prop}

\begin{proof}
The result follows by the direct calculation of the cohomology space, obtaining that $H^{2}_T\left( {\bf A},{\mathbb{F}}\right)$ is trivial for any of the terminal subalgebras ${\bf A}$ considered.
\end{proof}

Similarly, we can determine if there are central extensions for the rest of the identities mentioned in the previous section.

\begin{prop}
By calculating the corresponding cohomology space, we conclude the following.
\begin{enumerate}
\item There are no central extensions of $\overline{\mathcal{S}_{2,1}}$ in the variety defined by one identity from the proposition \ref{prop_id2}.

\item ${\rm dim}\, Z^{2}_P\left(\overline{ \mathcal{S}_{0,-3}},{\mathbb{F}}\right)=31$ and ${\rm dim}\, H^{2}_P\left( \overline{\mathcal{S}_{0,-3}},{\mathbb{F}}\right)=23$, where $P$ is the identity
$2 \mathfrak{st}^3_1 -3 \mathfrak{st}^3_2$ from proposition \ref{prop_id3}.

\item There are no central extensions of $S_2$ in the variety defined by one identity from the proposition \ref{prop_id6}.

\end{enumerate}
\end{prop}

Regarding the central extensions with respect to the identities $\mathfrak{st}^n_1$ and $\mathfrak{st}^n_2$ for $n=3,4,5$ (see Proposition \ref{prop_id5} and Proposition \ref{stprop}), we have the following result.

\begin{proposition}
\label{dims1}
The dimensions of the spaces of cocycles and coborders of the subalgebras of $W(2)$ are given.

{\begin{longtable}{|c|c|c|c|c|c|c|c|c|}

\hline
Algebra  & ${\rm dim}\, B^{2}$ & ${\rm dim}\, Z^{2}_{\mathfrak{st}^3_1}$ &  ${\rm dim}\,Z^{2}_{\mathfrak{st}^4_1}$ & ${\rm dim}\,Z^{2}_{\mathfrak{st}^5_1}$ & ${\rm dim}\,Z^{2}_{\mathfrak{st}^3_2}$ &  ${\rm dim}\,Z^{2}_{\mathfrak{st}^4_2}$ & ${\rm dim}\,Z^{2}_{\mathfrak{st}^5_2}$\\ \hline

$B_2$ &

7  &

- &

- &

44 &

- &

- &

44 \\ \hline
$W_2$ &

6  &

- &

- &

- &

- &

- &

30 \\ \hline
$C_2$ &

5  &

- &

- &

24 &

- &

- &

24 \\ \hline
$S_2$ &

4  &

- &

- &

16 &

- &

- &

16 \\ \hline
$D_2$ &

3  &

8 &

9 &

9 &

8 &

9 &

9 \\ \hline
$E_2$ &

2  &

4 &

4 &

4 &

4 &

4 &

4 \\ \hline

\end{longtable}}

\end{proposition}

\begin{proposition}
\label{dims2}
The dimensions of the spaces of cocycles and coborders of the subalgebras of $W(2)$ are given.

{\begin{longtable}{|c|c|c|c|c|c|c|c|c|}

\hline
Algebra & {\rm dim}\, $B^{2}$ & ${\rm dim}\, Z^{2}_{\mathfrak{st}^3_1}$ &
${\rm dim}\,Z^{2}_{\mathfrak{st}^4_1}$ & ${\rm dim}\,Z^{2}_{\mathfrak{st}^5_1}$ & ${\rm dim}\,Z^{2}_{\mathfrak{st}^3_2}$ &  ${\rm dim}\,Z^{2}_{\mathfrak{st}^4_2}$ & ${\rm dim}\,Z^{2}_{\mathfrak{st}^5_2}$\\ \hline

$W(2)$ &

8  &

- &

- &

- &

- &

- &

38 \\ \hline

$\overline{W(2)}$ &

8  &

- &

- &

54 &

- &

- &

54 \\ \hline

$\overline{\mathcal{S}_1}$ &

8  &

- &

- &

60 &

- &

- &

60 \\ \hline

$\overline{\mathcal{S}_5}$ &

8  &

- &

- &

60 &

- &

- &

60 \\ \hline

$\overline{\mathcal{S}_{-1, 1}}$ &

8  &

- &

- &

60 &

- &

- &

60 \\ \hline

$\overline{\mathcal{S}_{0, 0}}$ &

8  &

- &

- &

60 &

- &

- &

60 \\ \hline

$\widehat{W(2)}$ &

8  &

- &

- &

- &

- &

- &

40 \\ \hline

$\widehat{\widehat{W(2)}}$ &

8  &

- &

- &

50 &

- &

- &

54 \\ \hline

${\widetilde{W(2)}}$ &

8  &

- &

52 &

64 &

- &

52 &

64 \\ \hline

$\widetilde{\widetilde{W(2)}}$ &

8  &

43 &

64 &

64 &

43 &

64 &

64 \\ \hline

\end{longtable}
}
\end{proposition}

By Proposition \ref{dims1} and Proposition \ref{dims2}, we can conclude that for any of the subalgebras and contractions considered there are central extensions with respect to the identities $\mathfrak{st}^n_1$ and $\mathfrak{st}^n_2$ for $n=3,4,5$.

\appendix
\section*{Appendix A. Differences from the published version}
\addcontentsline{toc}{section}{Appendix A. Differences from the published version}

\noindent\textbf{Lemma~\ref{lem_codim1}:} the statement has been changed (the subalgebras ${\mathcal S}_5^{\gamma}$ with $\gamma\neq0$, ${\mathcal S}_7$, and ${\mathcal S}^{s}$ were missing) and the proof corrected.

\noindent\textbf{Proposition~\ref{prop_id3}:} the statement has been changed (the generator is $2\mathfrak{st}^3_1+3\mathfrak{st}^3_2$).

\end{document}